\documentclass[a4paper]{amsart}
\usepackage[frenchb]{babel}
\usepackage[utf8]{inputenc}
\usepackage[OT1]{fontenc}
\usepackage{import}
\usepackage{textcomp}
\usepackage[font=small,labelsep=none]{caption}
\usepackage[top=2.5cm,bottom=4cm,left=2.5cm,right=2.5cm]{geometry}
\usepackage{cite}
\usepackage{float}
\usepackage{hyperref}
\hypersetup{colorlinks,linktoc=all,citecolor=black,filecolor=black,linkcolor=blue,urlcolor=black
}
\usepackage{amsmath}
\usepackage{amsfonts}
\usepackage{amssymb}
\usepackage{amsthm}
\usepackage{pstricks}
\usepackage{tikz}
\tikzset{node distance=1.4cm, auto}
\usepackage{bbm}

\theoremstyle{plain}
\newtheorem{thm}{Théorème}[section]
\newtheorem{lem}[thm]{Lemme}
\newtheorem{prop}[thm]{Proposition}

\theoremstyle{definition}

\theoremstyle{remark}

\newcommand{\nep}[1]{#1_{N,\epsilon}}
\newcommand{\nepp}[1]{#1_{N,\epsilon'}}
\newcommand{\ep}[1]{#1_{\epsilon}}
\newcommand{\epp}[1]{#1_{\epsilon'}}
\newcommand{\expi}[1]{\exp(2\pi i #1)}
\newcommand{\expin}[1]{\exp(-2\pi i #1)}
\newcommand{\puie}[1]{\epsilon^{#1}}
\newcommand{\axne}[1]{#1(\alpha,x,N,\epsilon)}
\newcommand{\axnep}[1]{#1(\alpha,x,N,\epsilon')}
\newcommand{\axnke}[1]{#1(\alpha,x,N,k,\epsilon)}

\newcommand{\sset}{S(N,\alpha,\epsilon)}
\newcommand{\shset}{\hat{S}(N,\alpha,\epsilon)}
\newcommand{\kal}{\left\{k\alpha\right\}}
\newcommand{\opid}{\mathbbm{1}}
\newcommand{\puiea}{\puie{1+a+2a^2}}
\newcommand{\puieb}{\puie{1+\frac{2}{a}+2a}}
\newcommand{\torm}[1]{(\mathbb{T}^{1}-#1)\times \mathbb{T}^{1}}
\newcommand{\uset}{U(N,\alpha,\epsilon)}
\newcommand{\kv}{\vec{k}}
\newcommand{\mset}{m\in\mathbb{Z}^2-{0},\lVert m\rVert\leq M(\epsilon)}
\newcommand{\axnme}[1]{#1(\alpha,x,N,m,\epsilon)}
\newcommand{\limn}{\lim_{N\rightarrow\infty}}
\newcommand{\limin}{\liminf_{N\rightarrow\infty}}
\newcommand{\limsn}{\limsup_{N\rightarrow\infty}}
\newcommand{\sfn}{\frac{S(\alpha,x,N)}{N^a}}

\title{A Limit Theorem for Torus Translation}
\author{Linyuan Liu}

\address{Sydney Mathematical Research Institute\\
University of Sydney\\
NSW 2006\\
Australia}

\email{Linyuan.Liu@normalesup.org}
\begin{document}
\maketitle

Cet article est inspiré par la preuve du Théorème 10 de
\cite{DF15}. L'auteur remercie Bassam Fayad de lui avoir proposé cette
question et de lui avoir beaucoup aidée dans la préparation de cet
article. 

\section{Notations} 
On identifie $\tilde{M}=SL(2,\mathbb{R})/SL(2,\mathbb{Z})$ avec l'espace de réseaux sur $\mathbb{R}^2$.
Notons $\tilde{\mu}$ la mesure invariante sur $\tilde{M}$.
Soit $L\in\tilde{M}$, 
on note $e_1(L)$ le vecteur le plus court dans $L$ et $e_2(L)$ le vecteur le plus court 
dont la projection sur le complementaire orthogonal du sous-espace engendré par $e_1(L)$ est non-nulle. 
Alors pour $\tilde{\mu}$ presque tout $L\in\tilde{M}$, $e_1$,$e_2$ sont bien-définies.

Pour $m\in \mathbb{Z}^2$ et $L\in\tilde{M}$, on note $(m,e)$ le vecteur $m_1e_1(L)+m_2e_2(L)$ et $(X_m,Z_m)=(X_m(L),Z_m(L))$
son coordonnée.

On note $\lambda$ à la fois les mesures de Haar sur $\mathbb{T}^{1}$ et sur $\mathbb{T}^{2}$.
Notons $M=\tilde{M}\times \mathbb{T}^{2}$ et $\mu=\tilde{\mu}\times\lambda$ sur $M$.

\section{Énoncé du théorème principal}
Soient $a\in ]0,1[$ et $\alpha\in \mathbb{T}^{1}$. On définit
\begin{equation}
 S(\alpha,x,N)=\sum_{n=0}^{N-1}\varphi(x+n\alpha),
\end{equation}
où $\varphi(x)=\frac{1}{x^a}-\frac{1}{1-a}$ est de moyenne nulle.

Notre but est de démontrer le théorème suivant:
\begin{thm}\label{thm:principal}
  Il existe une fonction mesurable \( D \) sur $( M, \mu) $ qui est finie $\mu$-presque partout telle que pour Lebesgue
  presque tout $z\in \mathbb{R} $, on a
  \begin{equation}
    \lim_{ N\rightarrow \infty }\lambda \left\{ \left( \alpha,x
      \right)\in \mathbb{T}^{2}|\quad\frac{S(\alpha,x,N)}{N^a}\leq z
    \right\}=\mu\left\{ \left( L,\gamma \right)\in M\times \mathbb{T}^{2}|D\left(
    L,\gamma
    \right)\leq z
  \right\} .
    \label{eq:resultat}
  \end{equation}
\end{thm}

\section{Les termes non-résonants}
Définissons 
\begin{equation}
 \varphi_{N,\epsilon}(x) = 
 \begin{cases}
  \varphi(x)	& \qquad\text{si} \quad x\le\frac{\epsilon}{N}	\\
  \frac{1}{1-a}\left( \frac{N^a}{\epsilon^a}-1 \right) & \qquad \text{sinon}
 \end{cases}
\end{equation}
et 
\begin{equation}
 \axne{\triangle}=\frac{\sum_{n=0}^{N-1}\varphi_{N,\epsilon}(x+n\alpha)}{N^a}.
\end{equation}
Observons que $\int_{0}^{1}\varphi(x)dx=0$ et 
$\lambda\left\{(\alpha,x)\in\mathbb{T}^{2}|\quad \triangle(\alpha,x,N,\epsilon)\neq \frac{S(\alpha,x,N)}{N^a}\right\}\leq \epsilon$.

Si $k\in \mathbb{Z}^{*}$, 
\begin{eqnarray*}
\hat{\varphi}_{N,\epsilon}(k) &=&\int_{0}^{1}\exp(-2\pi ikx)\varphi_{N,\epsilon}(x)dx\\
   &=&\int_{\frac{\epsilon}{N}}^{1}\exp(-2\pi ikx)\frac{1}{x^a}dx+\int_{0}^{\frac{\epsilon}{N}}\frac{N^a}{(1-a)\epsilon^a}\exp(-2\pi ikx)dx\\
   &=&\frac{1}{1-a}[c_{N,\epsilon}(k)-i d_{N,\epsilon}(k)]-\frac{N^a}{(1-a)\epsilon^a}\frac{1}{2\pi ik}(\exp(-2\pi ik\frac{\epsilon}{N})-1).
\end{eqnarray*}

\begin{eqnarray*}
&&\nep{\hat{\varphi}}(k)\expi{kx}+\nep{\hat{\varphi}}(-k)\expin{kx}\\
&=&\left[\frac{2\nep{b}(k)}{k^{1-a}}+\frac{2N^a}{(1-a)\epsilon^a}\frac{1}{2\pi k}\sin(2\pi k \frac{\epsilon}{N})\right]\cos(2\pi kx)\\
&&+\left[\frac{2\nep{d}(k)}{k^{1-a}}-\frac{2N^a}{(1-a)\epsilon^a}\frac{1}{2\pi k}(\cos(2\pi k \frac{\epsilon}{N})-1)\right]\sin(2\pi kx),
\end{eqnarray*}
où
\begin{eqnarray}
 \nep{b}(k)&=&\int_{\frac{k\epsilon}{N}}^{k}\frac{\cos(2\pi x)}{x^a}dx\\
 \nep{d}(k)&=&\int_{\frac{k\epsilon}{N}}^{k}\frac{\sin(2\pi x)}{x^a}dx 
\end{eqnarray}

Donc
\begin{equation}
 \axne{\triangle}=\sum_{k=1}^{\infty}\axnke{g},
\end{equation}
où
\begin{eqnarray*}
 \axnke{g}&=&\left[\frac{2\nep{b}(k)}{k^{1-a}N^a}+\frac{2}{(1-a)\epsilon^a}\frac{\sin(2\pi k \frac{\epsilon}{N})}{2\pi k}\right]
 \frac{\cos\left[2\pi kx+\pi (N-1)k\alpha\right]\sin(\pi Nk\alpha)}{\sin(\pi k \alpha)}\\
 &&+\left[\frac{2\nep{d}(k)}{k^{1-a}N^a}-\frac{2}{(1-a)\epsilon^a}\frac{\cos(2\pi k \frac{\epsilon}{N})-1}{2\pi k}\right]
 \frac{\sin\left[2\pi kx+\pi (N-1)k\alpha\right]\sin(\pi Nk\alpha)}{\sin(\pi k \alpha)}\\
\end{eqnarray*}

\begin{lem}
$ \exists C >0$ ne dépendant que de $a$ tel que
 \begin{eqnarray}
  |\nep{b}(k)|&\leq& min\left\{C,C\frac{N^a}{k^a \epsilon^a}\right\}\\
  |\nep{d}(k)|&\leq& min\left\{C,C\frac{N^a}{k^a \epsilon^a}\right\}\\
  |\sin(2\pi k \frac{\epsilon}{N})|&\leq& min\left\{C,C\frac{k^a \epsilon^a}{N^a}\right\} \\
  |\cos(2\pi k \frac{\epsilon}{N})-1|&\leq& min\left\{C,C\frac{k^a \epsilon^a}{N^a}\right\} 
 \end{eqnarray}

\end{lem}

\begin{proof}
On peut le prouver par calculation simple. 
\end{proof}

Définissons 
\begin{equation}
 \axne{\bar{\triangle}}=\sum_{0<k<\frac{N}{\puie{1+2a}}}\axnke{g}.
\end{equation}

\begin{lem}
$ \exists C >0$ ne dépendant que de $a$ tel que 
 \begin{equation}
  \lVert\triangle(\cdotp,\cdotp,N,\epsilon)-\bar{\triangle}(\cdotp,\cdotp,N,\epsilon)\rVert_{L^2(\mathbb{T}^{2})}^2\leq C\epsilon.
 \end{equation}
\end{lem}

\begin{proof}
 Puisque $\left\{\cos\left[2\pi kx+\pi (N-1)k\alpha\right],\sin\left[2\pi k'x+\pi (N-1)k'\alpha\right]\right\}_{k,k'\in \mathbb{N}^*}$
 sont orthogonaux sur $L^2_x(\mathbb{T}^{1})$ et 
 \begin{equation*}
  \int_{\mathbb{T}^{2}}\left(\frac{\sin(\pi Nk\alpha)}{\sin(\pi k\alpha)}\right)^2 d\alpha\leq N,
 \end{equation*}
 on a 
 \begin{eqnarray*}
 &&\lVert\triangle(\cdotp,\cdotp,N,\epsilon)-\bar{\triangle}(\cdotp,\cdotp,N,\epsilon)\rVert_{L^2(\mathbb{T}^{2})}^2\\
 &\leq&\sum_{k\geq\frac{N}{\puie{1+2a}}}N\cdotp\left(C\frac{1}{k^{1-a}N^a}\cdotp\frac{N^a}{k^a\puie{a}}+C\frac{1}{k\puie{a}}\right)^2\\
 &\leq& C\frac{N}{\puie{2a}}\sum_{k\geq\frac{N}{\puie{1+2a}}}\frac{1}{k^2}\\
 &\leq& C\frac{N}{\puie{2a}}\frac{\puie{1+2a}}{N}=C\epsilon.
 \end{eqnarray*}
\end{proof}

Définissons 
\begin{equation}
 \axne{\tilde{\triangle}}=\sum_{k\in \sset}\axnke{g},
\end{equation}
où
\begin{equation}
 \sset=\left\{k\in\mathbb{N}^*|\quad k<\frac{N}{\puie{1+2a}},\quad k^{1-a}|\kal|\leq \frac{1}{\puie{1+a+2a^2}N^a}\right\},
\end{equation}
$\kal\in ]-\frac{1}{2},\frac{1}{2}]$ tel que $k\alpha-\kal\in\mathbb{Z}$. 

\begin{lem}
$ \exists C >0$ ne dépendant que de $a$ tel que 
 \begin{equation}
  \lVert\triangle(\cdotp,\cdotp,N,\epsilon)-\tilde{\triangle}(\cdotp,\cdotp,N,\epsilon)\rVert_{L^2(\mathbb{T}^{2})}^2\leq C\epsilon.
 \end{equation}
\end{lem}
\begin{proof}
Notons
\begin{equation*}
 B(k,p)=\left\{\alpha\in\mathbb{T}^{1}|\quad\leq k^{1-a}|\kal|<\frac{p+1}{\puie{1+a+2a^2}N^a}\right\}.
\end{equation*}
Alors $|B(k,p)|\leq \frac{1}{k^{1-a}\puiea N^a}$.
 \begin{eqnarray*}
 &&\lVert\bar{\triangle}(\cdotp,\cdotp,N,\epsilon)-\tilde{\triangle}(\cdotp,\cdotp,N,\epsilon)\rVert_{L^2(\mathbb{T}^{2})}^2\\
 &\leq&\sum_{k<\frac{N}{\puie{1+2a}}}\left(\frac{C}{k^{1-a}N^a}+\frac{C\frac{k^a\puie{a}}{N^a}}{\puie{a}k}\right)^2 \sum_{p=1}^{\infty}\int_{\mathbb{T}^{1}}\frac{1}{\kal^2}\opid_{B(k,p)}d\alpha\\
 &\leq&\sum_{k<\frac{N}{\puie{1+2a}}}\frac{C}{k^{2-2a}N^{2a}}\sum_{p=1}^{\infty}\frac{k^{1-a}\puie{1+a+2a^2}N^a}{p^2
 }\\
 &\leq&\frac{C\puie{1+a+2a^2}}{N^a}\sum_{k<\frac{N}{\puie{1+2a}}}\frac{1}{k^{1-a}}\\
 &\leq&\frac{C\puie{1+a+2a^2}}{N^a}\cdotp\frac{N^a}{\puie{a+2a^2}}=C\epsilon.
 \end{eqnarray*}
\end{proof}

Définissons
\begin{equation}
 \axne{\hat{\triangle}}=\sum_{k\in \shset}\axnke{g},
\end{equation}
où
\begin{equation}
 \shset=\left\{k\in\mathbb{N}^*|\quad \puieb N<k<\frac{N}{\puie{1+2a}},\quad k^{1-a}|\kal|\leq \frac{1}{\puie{1+a+2a^2}N^a}\right\}.
\end{equation}

Posons
\begin{eqnarray*}
 E_{k,N,\epsilon}&=&\left\{\alpha\in\mathbb{T}^{1}|\quad k^{1-a}|\kal|\leq\frac{1}{\puiea N^a}\right\}\\
 E_{N,\epsilon}&=&\bigcup_{k\leq \puieb N}E_{k,N,\epsilon}.
\end{eqnarray*}
On a 
\begin{equation}
\lVert\triangle(\cdotp,\cdotp,N,\epsilon)-\hat{\triangle}(\cdotp,\cdotp,N,\epsilon)\rVert_{L^2(\torm{E_{N,\epsilon}})}^2\leq C\epsilon
\qquad \text{et}\qquad |E_{N,\epsilon}|\leq C\epsilon.
\end{equation}

Définissons
\begin{equation}
 \axne{\diamondsuit}=\sum_{k\in\shset}\axnke{\tilde{g}},
\end{equation}
où
\begin{eqnarray*}
\axnke{\tilde{g}}&=&\left[\frac{2b}{k^{1-a}N^a}+\frac{2}{(1-a)\epsilon^a}\frac{\sin(2\pi k \frac{\epsilon}{N})}{2\pi k}\right]
 \frac{\cos\left[2\pi kx+\pi (N-1)k\alpha\right]\sin(\pi Nk\alpha)}{\pi \kal}\\
 &&+\left[\frac{2d}{k^{1-a}N^a}-\frac{2}{(1-a)\epsilon^a}\frac{\cos(2\pi k \frac{\epsilon}{N})-1}{2\pi k}\right]
 \frac{\sin\left[2\pi kx+\pi (N-1)k\alpha\right]\sin(\pi Nk\alpha)}{\pi\kal}\\
b&=&\int_{0}^{\infty}\frac{\cos(2\pi x)}{x^a}dx\\
d&=&\int_{0}^{\infty}\frac{\sin(2\pi x)}{x^a}dx\\ 
\end{eqnarray*}

En calculant, on obtient
\begin{eqnarray*}
&&\lVert\hat{\triangle}(\cdotp,\cdotp,N,\epsilon)-\diamondsuit(\cdotp,\cdotp,N,\epsilon)\rVert_{L^2(\torm{E_{N,\epsilon}})}^2 \\
&\leq&C\puie{1-4a^2}\left(\frac{\sin \theta_N}{\theta_N}-1\right)^2+\frac{C}{N^2a\puie{3+2a+8a^2}}\\
&&\longrightarrow 0 \qquad \text{lorsque}\quad N\rightarrow\infty,
\end{eqnarray*}
où $\theta_N=\frac{1}{\puie{\frac{2}{a}+2a}N}$.

Si $\uset$ est une partie de $\mathbb{N}^*$ contenant $\shset$ et
\begin{equation}
 \axne{\bar{\diamondsuit}}=\sum_{k\in\uset}\axnke{\tilde{g}},
\end{equation}

alors on aura
\begin{equation}
\varlimsup_{N\rightarrow\infty}\lVert\triangle(\cdotp,\cdotp,N,\epsilon)-\bar{\diamondsuit}(\cdotp,\cdotp,N,\epsilon)\rVert_{L^2(\torm{E_{N,\epsilon}})}^2\leq C\epsilon.
\end{equation}

\section{Les termes oscillants}
Notons,
\begin{displaymath}
  g_{T}=
  \big(\begin{smallmatrix}
    e^{-T}& 0\\
    0& e^{T}
  \end{smallmatrix}\big),\quad \Lambda_{\alpha}=
  \big(\begin{smallmatrix}
    1&0\\
    \alpha&1
  \end{smallmatrix}\big). 
\end{displaymath}

Considérons le réseaux $L(N,\alpha)=g_{\ln N}\Lambda_{\alpha}\mathbb{Z}^2$. On associe à chaque $k\in\mathbb{Z}$ un vecteur
$\kv=(k,k')$ où $k'=k'(k,\alpha)$ l'unique entier tel que $|\kal|=|k\alpha+k'|\leq \frac{1}{2}$ et $k\alpha+k'\neq -\frac{1}{2}$.
Notons $(X,Z)=g_{\ln N}\Lambda_{\alpha}\kv$. 
Observons que $k\in\shset$ si et seulement si 
\begin{equation*}
 \puieb<X<\frac{1}{\puie{1+2a}},\quad |Z|<\frac{1}{X^{1-a}\puiea}.
\end{equation*}

Notons $e_i(N,\alpha)=e_i(L(N,\alpha))$. On aura besoin la Proposition
suivante (voir \cite{DF14} Proposition 4.3).

\begin{prop}
 \label{thm:unidis}
 Soit $\Phi:(\mathbb{R}^2)^2\times\mathbb{R}\rightarrow\mathbb{R}$ continue et bornée. Alors on a,
 \begin{equation}
  \limn\int\Phi(e_1(N,\alpha),e_2(N,\alpha),\alpha)d\alpha
 =\int_{\tilde{M}\times\mathbb{T}^{1}}\Phi(e_1(L),e_2(L),\alpha)d\tilde{\mu}(L)d\alpha.
 \end{equation}

\end{prop}

\begin{prop}
$\forall \epsilon>0\quad\exists M(\epsilon)>0$ tel que si $\alpha\notin\nep{E}$ 
alors $k\in\shset$ implique que
\begin{equation*}
 g_{\ln N}\Lambda_{\alpha}\kv=m_1e_1(N,\alpha)+m_2e_2(N,\alpha)
\end{equation*}
pour un unique $(m_1,m_2)\in \mathbb{Z}^{d+1}-{(0,0)}$, $\lVert m\rVert\leq M(\epsilon)$ et $X_m>0$.

Si on fixe $\epsilon>0$ et $N$ est assez grand et $\alpha\notin \nep{E}$, 
alors pour tout $\lVert m\rVert\leq M(\epsilon)$, il existera un unique $\kv\in\mathbb{Z}^2$ tel que
\begin{equation*}
 g_{\ln N}\Lambda_{\alpha}\kv=(m,e(N,\alpha))=m_1e_1(N,\alpha)+m_2e_2(N,\alpha).
\end{equation*}

On note $\uset$ l'ensemble de $k\in\mathbb{N}^*$ correspondant au l'ensemble de $m\in\mathbb{Z}^2-{(0,0)}$,$\lVert m\rVert\leq M(\epsilon)$ et $X_m>0$.
\end{prop}
\begin{proof}
Par la définition de $\shset$, on sait que si $k\in \shset$ alors 
$\lVert g_{\ln N}\Lambda_{\alpha}\kv\rVert\leq \puie{-1-2a}+\puie{-a^2-a-\frac{2}{a}}=R(\epsilon)$.
Comme les normes $\lVert x\rVert$ et $\lVert x_1e_1(T)+x_2e_2(T)\rVert$ sont équivalantes, pour tout $L\in\tilde{M}$ il existe $M(L)$
tel que $\lVert m_1e_1(T)+m_2e_2(T)\rVert\geq R(\epsilon)$ une fois $\lVert m\rVert\geq M(L)$. Pour montrer que $M(L)$ peut être choisi uniformément, il suffit de montrer que
\begin{equation*}
 \{L(N,\alpha)|\quad \alpha\notin\nep{E}\}
\end{equation*}
est précompact. Par la définition de $\nep{E}$, si $|X|\leq \puieb$ alors $|Z|>\frac{1}{X^{1-a}\puiea}$. Donc $\exists \delta(\epsilon)>0$ tel que si $\alpha\notin\nep{E}$ alors 
tout vecteur dans $L(N,\alpha)$ est plus long que $\delta(\epsilon)$. On déduit la précompacité du critère de la compacité de Mahler.

Pour tout $\lVert m\rVert\leq M(\epsilon)$, il existe $\bar{k}\in\mathbb{Z}^2$ tel que $(m,e(N,\alpha))= g_{\ln N}\Lambda_{\alpha}\bar{k}$. Pour $N$ suffisament grand,
on a $\lVert (m,e(N,\alpha))\rVert\leq \frac{1}{4} N$,$\forall\lVert m\rVert\leq M(\epsilon)$ et $\alpha\notin\nep{E}$. Donc $\bar{k}=\kv(k)$ dans ce cas.
\end{proof}

On en déduit que 
\begin{equation}
 \axne{\Box}= \axne{\bar{\diamondsuit}}\qquad \text{sur}\quad\torm{\nep{E}},
\end{equation}
où
\begin{eqnarray*}
 \axne{\Box}&=&\sum_{\mset}\axnme{h},\\
 \axnme{h}&=&\left[\frac{2b}{X_m^{1-a}}+\frac{2}{(1-a)\epsilon^a}\frac{\sin(2\pi X_m\epsilon)}{2\pi X_m}\right]
 \frac{\cos\left[2\pi m\cdotp \gamma+\pi \frac{N-1}{N}Z_m\right]\sin(\pi Z_m)}{\pi Z_m}\opid_{]0,\infty[}(X_m)\\
 &+&\left[\frac{2d}{X_m^{1-a}}-\frac{2}{(1-a)\epsilon^a}\frac{\cos(2\pi X_m\epsilon)-1}{2\pi X_m}\right]
 \frac{\sin\left[2\pi m\cdotp \gamma+\pi \frac{N-1}{N}Z_m\right]\sin(\pi Z_m)}{\pi Z_m}\opid_{]0,\infty[}(X_m),\\
 \gamma&=&\gamma(\alpha,x,N)=(Nxe_{11}(N,\alpha),Nxe_{21}(N,\alpha)).
\end{eqnarray*}

Définissons
\begin{eqnarray*}
 \ep{D}(L,\gamma)&=&\sum_{\mset} q(L,\gamma,m,\epsilon),\\
q(L,\gamma,m,\epsilon)&=&\left[\frac{2b}{X_m^{1-a}}+\frac{2}{(1-a)\epsilon^a}\frac{\sin(2\pi X_m\epsilon)}{2\pi X_m}\right]
 \frac{\cos\left[2\pi m\cdotp \gamma+\pi Z_m\right]\sin(\pi Z_m)}{\pi Z_m}\opid_{]0,\infty[}(X_m)\\
 &&+\left[\frac{2d}{X_m^{1-a}}-\frac{2}{(1-a)\epsilon^a}\frac{\cos(2\pi X_m\epsilon)-1}{2\pi X_m}\right]
 \frac{\sin\left[2\pi m\cdotp \gamma+\pi Z_m\right]\sin(\pi Z_m)}{\pi Z_m}\opid_{]0,\infty[}(X_m). 
\end{eqnarray*}

\begin{prop}
 Si $\alpha,x$ sont uniformément distribués dans $\mathbb{T}^{2}$, alors les variables aléatoires
 \begin{equation}
  e_1(N,\alpha),e_2(N,\alpha), \gamma
 \end{equation}
 convergent dans distribution ver $\mu$ lorsque $N\rightarrow\infty$.
 
 En particulier, la distribution de $\Box(\cdotp,\cdotp,N,\epsilon)$ converge vers la distribution de $\ep{D}$ lorsque $N\rightarrow\infty$.

\end{prop}

\begin{proof}
Il suffit de montrer que $\forall n_1,n_2$ entiers et $\forall \Phi$ à support compact, on a
\begin{eqnarray*}
 &&\limn\iint\Phi(e_1(N,\alpha),e_2(N,\alpha))\expi{(n_1\gamma_1+n_2\gamma_2)}dxd\alpha\\
 &=&\int_{\tilde{M}}\Phi(e_1(L),e_2(L))d\tilde{\mu}(L)\int_{\mathbb{T}^{2}}\expi{(n_1\gamma_1+n_2\gamma_2)}d\gamma.
\end{eqnarray*}
Si $n_1=n_2=0$, c'est un résultat direct de la Proposition \ref{thm:unidis}. Donc on peut supposer que $n_1\neq0$ ou $n_2\neq0$, et on doit prouver que
\begin{equation} \label{eq:dis}
\iint\Phi(e_1(N,\alpha),e_2(N,\alpha))\expi{(n_1\gamma_1+n_2\gamma_2)}dxd\alpha\rightarrow0 
\end{equation}
Comme pour presque tout $L$, $e_{1,1}$ et $e_{2,1}$ sont indépendants
sur $\mathbb{Z}$, la Proposition \ref{thm:unidis} nous dit que
\begin{equation*}
 \lambda(\alpha\in\mathbb{T}^{1}|\quad \lvert n_1e_{1,1}(N,\alpha)+n_2e_{2,1}(N,\alpha)\rvert<\frac{1}{N^{\frac{1}{2}}})\rightarrow0
\end{equation*}
lorsque $N\rightarrow\infty$. Donc LHS de \ref{eq:dis}$=I+II$ où $I$ est l'intégrale sur $\lvert n_1e_{1,1}+n_2e_{2,1}\rvert<\frac{1}{N^{\frac{1}{2}}}$
et $II$ est l'intégrale sur $\lvert n_1e_{1,1}+n_2e_{2,1}\rvert\geq\frac{1}{N^{\frac{1}{2}}}$. On a donc
\begin{eqnarray*}
 |I|&\leq &\lVert\Phi\rVert_{L^{\infty}}\lambda(\alpha\in\mathbb{T}^{1}|\quad \lvert n_1e_{1,1}(N,\alpha)+n_2e_{2,1}(N,\alpha)\rvert<\frac{1}{N^{\frac{1}{2}}})\rightarrow0.\\
 |II|&\leq&\frac{\lVert\Phi\rVert_{L^{\infty}}}{N^{\frac{1}{2}}}\rightarrow0.
\end{eqnarray*}

\end{proof}

\section{La fin de la démonstration}
\begin{lem}\label{lem:ex}
 La famille $\left\{\ep{D}\right\}_{\epsilon>0}$ sur $(M,\mu)$ convergent par rapport à la mesure vers une fonction
 $D$ mesurable et finie $\mu$-presque partout.
\end{lem}
\begin{proof}
$\forall \delta>0$ et $\forall \epsilon,\epsilon'>0$, on a
\begin{eqnarray*}
 &&\mu\left\{(L,r)\in M|\quad |\ep{D}(L,\gamma)-\epp{D}(L,\gamma)|>\delta\right\}\\
 &=&\limn \lambda\left\{(\alpha,x)\in\mathbb{T}^{2}|\quad|\axne{\Box}-\axnep{\Box}|>\delta\right\}\\
 &\leq&\limin\lambda\left\{(\alpha,x)\in\torm{\nep{E}-\nepp{E}}|\quad|\axne{\Box}-\axnep{\Box}|>\delta\right\}\\
 &&+\limin \lambda (\nep{E}\cup\nepp{E})\\
 &\leq&\limin\lambda\left\{(\alpha,x)\in\torm{\nep{E}-\nepp{E}}|\quad|\axne{\Box}-\axne{\triangle}|>\frac{\delta}{3}\right\}\\
 &&+\limin\lambda\left\{(\alpha,x)\in\torm{\nep{E}-\nepp{E}}|\quad|\axne{\triangle}-\axnep{\triangle}|>\frac{\delta}{3}\right\}\\
 &&+\limin\lambda\left\{(\alpha,x)\in\torm{\nep{E}-\nepp{E}}|\quad|\axnep{\Box}-\axnep{\triangle}|>\frac{\delta}{3}\right\}\\
 &&+C(\epsilon+\epsilon')\\
 &\leq& C\frac{\epsilon+\epsilon'}{\delta^2}+C(\epsilon+\epsilon')<\delta,\quad \text{si} \quad \epsilon+\epsilon'<\frac{\delta^3}{C(1+\delta^2)},
\end{eqnarray*}
d'où la famille $\left\{\ep{D}\right\}_{\epsilon>0}$ est Cauchy par rapport à la mesure. Donc elle converge vers une fonction $D$ sur $(M,\mu)$.

Comme les $\ep{D}$ sont finies presque partout, $D$ l'est aussi.

\end{proof}

\begin{prop}
La fonction $D$ obtenue dans lemme \ref{lem:ex} vérifie les conditions dans théorème \ref{thm:principal}.
\end{prop}

\begin{proof}
$\forall \epsilon>0$, $\forall \delta>0$, on a
\begin{eqnarray*}
 &&\lambda\left\{(\alpha,x)\in \mathbb{T}^{2}|\quad\sfn\leq z\right\}\\
 &\leq&\lambda\left\{(\alpha,x)\in \mathbb{T}^{2}|\quad\axne{\triangle}\leq z\right\}+C\epsilon\\
 &\leq&\lambda\left\{(\alpha,x)\in \torm{\nep{E}}|\quad\axne{\triangle}\leq z\right\}+C\epsilon\\
 &\leq&\lambda\left\{(\alpha,x)\in \torm{\nep{E}}|\quad\axne{\Box}\leq z+\delta\right\}\\
 &&+\lambda\left\{(\alpha,x)\in \torm{\nep{E}}|\quad|\axne{\triangle}-\axne{\Box}|\geq \delta\right\}+C\epsilon\\
 &\leq&\lambda\left\{(\alpha,x)\in \mathbb{T}^{2}|\quad\axne{\Box}\leq z+\delta\right\}\\
 &&+\frac{\lVert\Box-\triangle\rVert_{L^2(\torm{\nep{E}})}^2}{\delta^2}+C\epsilon.
\end{eqnarray*}
Donc,
\begin{eqnarray*}
 &&\limsn\lambda\left\{(\alpha,x)\in \mathbb{T}^{2}|\quad\sfn\leq z\right\}\\
 &\leq&\mu\left\{(L,\gamma)\in M|\quad \ep{D}(L,\gamma)\leq z+\delta\right\}+\frac{C\epsilon}{\delta^2}+C\epsilon\\
 &\leq&\mu\left\{(L,\gamma)\in M|\quad D(L,\gamma)\leq z+2\delta\right\}\\
 &&+\mu\left\{(L,\gamma)\in M|\quad |\ep{D}(L,\gamma)-D(L,\gamma)|\geq \delta\right\}+\frac{C\epsilon}{\delta^2}+C\epsilon\\
 &&\longrightarrow^{\epsilon\rightarrow 0}\mu\left\{(L,\gamma)\in M|\quad D(L,\gamma)\leq z+2\delta\right\}\\
 &&\longrightarrow^{\delta\rightarrow 0}\mu\left\{(L,\gamma)\in M|\quad D(L,\gamma)\leq z\right\}
\end{eqnarray*}

De même façon, on déduit que
\begin{equation*}
 \limin\lambda\left\{(\alpha,x)\in \mathbb{T}^{2}|\quad\sfn\leq z\right\}\geq\mu\left\{(L,\gamma)\in M|\quad D(L,\gamma)< z\right\}.
\end{equation*}

\end{proof}

\def\refname{References}

\end{document}